\documentclass[12pt]{amsart}
\usepackage{amsmath}
\usepackage{amsthm}
\usepackage{latexsym}
\usepackage{amssymb}
\usepackage{mathrsfs}
\DeclareMathOperator{\Char}{Char}

\DeclareMathOperator{\re}{Re}

\DeclareMathOperator{\ad}{ad}

\DeclareMathOperator{\supp}{supp}

\newcommand{\dbar}{\ \ \mathchar'26\mkern-11mu d}
\newtheorem{theorem}{Theorem}
\newtheorem{proposition}{Proposition}
\newtheorem{lemma}{Lemma}
\newtheorem{corollary}{Corollary}

\newtheorem{remark}{Remark}
\newtheorem{conjecture}{Conjecture}
\def\jpopn#1#2{%
  \mathopen{%
    \setbox0=\hbox{$#1\langle$}%
    \setbox2=\hbox{%
            {\hbox{$#1\langle$}}%
            \kern -.6\wd0\box0%
    }%
            \box2%
  }%
}

\def\jpcls#1#2{%
  \mathclose{%
    \setbox0=\hbox{$#1\rangle$}%
    \setbox2=\hbox{%
            {\hbox{$#1\rangle$}}%
            \kern -.6\wd0\box0%
    }%
            \box2%
  }%
}

\def\tp{\ {}^{t}\kern -3pt}

\renewcommand{\thetheorem}{\arabic{theorem}}
\renewcommand{\theproposition}{\thesection.\arabic{proposition}}
\renewcommand{\thelemma}{\thesection.\arabic{lemma}}
\renewcommand{\thedefinition}{\thesection.\arabic{definition}}
\renewcommand{\thecorollary}{\thesection.\arabic{corollary}}
\renewcommand{\theequation}{\thesection.\arabic{equation}}
\renewcommand{\theremark}{\thesection.\arabic{remark}}

\begin{document}
%
\def\A {{\mathcal{A}}}
\def\D {{\mathcal{D}}}
\def\R {{\mathbb{R}}}
\def\N {{\mathbb{N}}}
\def\C {{\mathbb{C}}}
\def\Z {{\mathbb{Z}}}
\def\Q {{\mathbb{Q}}}
\def\phi{\varphi}
\def\epsilon{\varepsilon}
\def\kappa{\varkappa}
\def\tb#1{\|\kern -1.5pt | #1 \|\kern -1.5pt |}
%
\setcounter{tocdepth}{1}
%
%

\title[Treves Conjecture]{Analytic Hypoellipticity for Sums of
  Squares and the Treves Conjecture}
\author{Paolo Albano}
\address{Dipartimento di Matematica, Universit\`a
di Bologna, Piazza di Porta San Donato 5, Bologna
Italy}
\email{paolo.albano@unibo.it} 
\author{Antonio Bove}
\address{Dipartimento di Matematica, Universit\`a
di Bologna, Piazza di Porta San Donato 5, Bologna
Italy}
\email{bove@bo.infn.it}
\author{Marco Mughetti}
\address{Dipartimento di Matematica, Universit\`a
di Bologna, Piazza di Porta San Donato 5, Bologna
Italy}
\email{marco.mughetti@unibo.it}
\date{\today}

\begin{abstract}
We are concerned with the problem of real analytic regularity of the
solutions of sums of squares with real analytic coefficients. Treves
conjecture states that an operator of this type is analytic
hypoelliptic if and only if all the strata in the Poisson-Treves
stratification are symplectic. 

We produce a model operator, $ P_{1} $, having a single symplectic
stratum and prove that it is Gevrey $ s_{0} $ hypoelliptic and not
better. See Theorem \ref{th:1} for a definition of $ s_{0} $. We also
show that this phenomenon has a microlocal character. 

We point out explicitly that this is a counterexample to the
sufficient part of Treves conjecture and not to the necessary part,
which is still an open problem.
\end{abstract}
\maketitle
%

\section{Introduction}
\renewcommand{\thetheorem}{\thesection.\arabic{theorem}}

This paper is concerned with the real analytic regularity of the
distribution solutions to equations of the form
\begin{equation}
\label{eq:P}
P (x, D) u = \sum_{j=1}^{N} X_{j}(x, D)^{2} u = f ,
\end{equation}
where $ X_{j}(x, D) $ is a vector field with real analytic
coefficients defined in an open set $ \Omega \subset \R^{n} $, $ u $
is a distribution in $ \Omega $ and $ f \in C^{\omega}(\Omega) $, the
space of all real analytic functions in $ \Omega $. 

The problem of the $ C^{\infty}(\Omega) $ hypoellipticity of
\eqref{eq:P} has been solved completely by L. H\"ormander in 1967,
\cite{hormander67}, by proving that $ P $ is hypoelliptic if the
vector fields defining it verify the condition
\begin{itemize}
\item[(H)]{}
The Lie algebra generated by the vector fields and their commutators has
dimension $ n $, equal to the dimension of the ambient space.
\end{itemize}
We note in passing that M. Derridj, \cite{derridj71}, showed that when
the vector fields have real analytic coefficients condition (H) is
also necessary for $ C^{\infty} $ hypoellipticity. 

As a further step in studying the hypoellipticity properties of $ P $
one may ask if, when condition (H) is satisfied, it is analytic
hypoelliptic, i.e. if $P u = f$, $ f \in 
C^{\omega}(\Omega)$, for a certain distribution $ u \in
\mathscr{D}'(\Omega) $, implies that actually $u \in
C^{\omega}(\Omega)$. 

In 1972 M. S. Baouendi and C. Goulaouic produced an example of a sum
of squares satisfying condition (H)---and hence $ C^{\infty} $
hypoelliptic---which is not analytic hypoelliptic. Their model was
followed by another found almost simultaneously by O.~A.~Ole\u \i nik
and O.~A.~Ole\u \i nik and E.~V.~Radkevi\v c in \cite{o72},
\cite{ol-rad-73}, that is not analytic hypoelliptic and
satisfies condition (H). The difference between the two models is
important, even though at the time was not completely understood.

In 1978 F. Treves and in 1980 D. S. Tartakoff independently published the
papers \cite{treves78} and \cite{Tartakoff1980}, where, using
very different methods of proof, they proved the analytic
hypoellipticity for sums of squares that satisfy condition (H), vanish
to the exact order two on the characteristic variety, $ \Char(P) $,
when $ \Char(P) $ is actually a symplectic real analytic manifold.

In 1980 M\'etivier, following the ideas of F. Treves, showed that if
$P$ is an analytic 
(pseudo)differential operator of symbol $p(x, \xi) = p_{m} (x, \xi) +
p_{m - 1} (x, \xi) + \cdots$, where $p_{m}$ vanishes exactly to the
order $m$ on a symplectic real analytic manifold and $p_{m - j}$
vanishes at least to the order $(m - 2j)_{+}$ on the same manifold,
then, if $P$ is $ C^{\infty} $ hypoelliptic, it is also analytic
hypoelliptic. The vanishing conditions on the 
lower order terms are also called the Levi conditions. A few years
later Okaji, \cite{okaji}, generalized M\'etivier's work, arguing
along the same 
lines, to a case where the vanishing is anisotropic (but exact and
with Levi conditions) on a symplectic characteristic manifold.

These papers seemed to imply that a symplectic characteristic manifold
should be necessary for analytic hypoellipticity. 
Actually Treves in \cite{treves78} conjectured that ``if $ \Char(P) $,
assumed to be an analytic manifold, contains a smooth curve which is
orthogonal for the fundamental symplectic form to the whole tangent
plane to $ \Char(P) $ at every point (of the curve), the operator $ P
$ might not be analytic hypo-elliptic. Actually it is my belief that,
in this case, $ P $ is necessarily not so.''

In 1998 Hanges and Himonas, \cite{hh98}, observed that the Ole\u \i
nik and Radkevi\v c operator has a real analytic symplectic
characteristic manifold, and thus no Treves' curves, but is not
analytic hypoelliptic. Therefore the absence of Treves curves does not
imply analytic hypoellipticity. 


To overcome this difficulty, in 1996, \cite{Treves}, F. Treves stated
a conjecture for the sums of squares of vector fields that takes into
account all the cases known to that date. The conjecture requires some
work to be stated precisely; see to this end the papers
\cite{Treves}, \cite{btreves}, \cite{trevespienza}. In what follows we
give a brief, sketchy account of how to formulate it.

Let $ P $ be as in \eqref{eq:P}. Then the characteristic variety of $
P $ is $ \Char(P) = \{ (x, \xi) \ | \ X_{j}(x, \xi) = 0, j=1, \ldots,
N \} $. This is a real analytic variety and, as such, it can be
stratified in real analytic manifolds, $ \Sigma_{i} $, for $ i $ in a
family of indices, having the property that for $ i \not = i' $,
either $ \Sigma_{i} \cap \overline{\Sigma}_{i'} = \varnothing $ or, if $
\Sigma_{i} \cap \overline{\Sigma}_{i'} \not = \varnothing $, then $ \Sigma_{i}
\subset \partial \Sigma_{i'} $. We refer to \cite{treves-book} for
more details.

Next one examines the rank of the restriction of the
symplectic form to the analytic strata $ \Sigma_{i} $. If there is a
change of rank of the symplectic form on a stratum, we may add to the
equations of the stratum the equations of the subvariety where there
is a change in rank and stratify the so obtained variety into strata
which are real analytic manifolds where the symplectic form has
constant rank.  

In the final step one considers the multiple Poisson brackets of the
symbols of the vector fields. Denote by $ X_{j}(x, \xi) $ the symbol
of the $ j $-th vector field. Let $ I = (i_{1}, i_{2}, \ldots, i_{r})
$, where $ i_{j} \in \{1, \ldots, N\} $. Write $ |I| = r $ and define
$$ 
X_{I}(x, \xi) = \{ X_{i_{1}}(x, \xi)  , \{ X_{i_{2}}(x, \xi)  , \{
\cdots \{ X_{i_{r-1}}(x, \xi), X_{i_{r}}(x, \xi) \} \cdots \} \} \}.
$$
$ r $ is called the length of the multiple Poisson bracket $ X_{I}(x,
\xi) $. For each stratum previously obtained, say $ \Sigma_{ik} $, we
want that all brackets of length lesser than a certain integer, say $
\ell_{ik} $ vanish, but that there is at least one bracket of length $
\ell_{ik} $ which is non zero on $ \Sigma_{ik} $. One can show that
this makes sense and defines a stratification.

Thus the strata obtained are real analytic manifolds where the
symplectic form has constant rank and where all brackets of the vector
fields vanish if their length is $ < \ell_{ik} $, and there is at
least one microlocally elliptic bracket of length $ \ell_{ik} $, $
\ell_{ik} $ depending on the stratum. $ \ell_{ik} $ is also called the
depth of the stratum.

We now state Treves' conjecture:
\begin{conjecture}[Treves, \cite{Treves}]
The operator $ P $ in \eqref{eq:P} is analytic hypoelliptic if and
only if every stratum in the above described stratification is
symplectic.
\end{conjecture}
We remark that the above statement is in agreement with a number of
known results and that no counterexamples are known. 
Baouendi-Goulaou\-ic operator does not have a symplectic
characteristic manifold and so one might expect it not to be analytic
hypoelliptic. In fact one can prove that it propagates Gevrey $ s $
singularities along the Hamilton leaves of the characteristic
manifold, for $ 1 \leq s < 2 $. 

Consider the Ole\u \i nik-Radkevi\v c operator
\begin{equation}
\label{eq:ol-rad-op}
D_{1}^{2} + x_{1}^{2(p-1)} D_{2}^{2} + x_{1}^{2(q-1)} D_{3}^{2},
\end{equation}
where $ 1 < p \leq q $. It is straightforward to see that its
characteristic manifold is the manifold $ \{ (x, \xi) \ | \ x_{1} =
\xi_{1} = 0, (\xi_{2}, \xi_{3}) \neq (0, 0) \} \subset T^{*}\R^{3}
\setminus \{0\} $. It is also evident that the latter is a symplectic
submanifold of the cotangent bundle. Here we need to look closely at
the Poisson brackets: let us consider $ \{ X_{1}, \{ X_{1} , \{ \ldots \{
X_{1}, X_{2}\} \ldots \}\}\} = \ad_{X_{1}}^{p-1}X_{2} = (p-1)! \xi_{2}
$. The latter is not elliptic in a conic neighborhood of the point $
(0, 0, 0; 0, 0, 1) $. Hence, if $ p < q $, the strata are given by  $
\{ (x, \xi) \ | \ x_{1} = 
\xi_{1} = 0, \xi_{2} \neq 0 \}  $ and $ \{ (x, \xi) \ | \ x_{1} =
\xi_{1} = 0, \xi_{2} = 0, \xi_{3} \neq 0 \}  $, which is not
symplectic. 

Treves and Tartakoff theorems are also in agreement with
the conjecture, since in that case there is only one symplectic
stratum and the vanishing is exactly of order two.

We just would like to mention that a number of results have been
published over the last fifteen years in agreement with the
conjecture. As a non exhaustive and certainly incomplete list we
mention the papers \cite{christ}, \cite{ch09}, \cite{costin-costin},
\cite{grigis-sjostrand}, \cite{grushin}, \cite{hh}, \cite{hoshiro} as
well as \cite{AB}, \cite{BT}.

\bigskip
In the present paper we exhibit an operator whose stratification has
just a single symplectic stratum and we show that the operator is
Gevrey $ s_{0} $ hypoelliptic and that this is the optimal Gevrey
regularity (see Theorem \ref{th:1} for a precise definition of $ s_{0}
$.)

As a consequence we get that a symplectic stratification does not
imply analytic hypoellipticity. In our opinion
this prompts for a revision of the sufficient part of the conjecture,
mainly where the Poisson brackets are involved.

The necessary part of the conjecture, as far as we know, is still an
open problem: If there is a non symplectic stratum, so that
Hamilton leaves appear, then the operator $ P $ is not analytic
hypoelliptic. 

Here is the structure of the paper. In Section 2 we state the result
for two operators: one having a single simplectic stratum and the
other having a stratification where the less deep stratum is
symplectic. For the first operator we get a local Gevrey regularity,
while for the second the regularity is only at a microlocal level. See
\cite{trevespienza} for a microlocal statement of the conjecture. 

Section 3 is devoted to the proof of the optimality of the $ s_{0} $
Gevrey regularity. We construct a solution to $ P_{1}u = 0 $ which is
not better than Gevrey $ s_{0} $. The solution we construct has a complex phase
function (and this is a well known phenomenon) as well as an
amplitude. To obtain the amplitude we have to discuss a semiclassical
spectral problem for a stationary Schr\"odinger equation with a
symmetric double well potential depending on two parameters.

It is known that, since the bottom of the well is quadratic, for very
small values of the Planck constant $ h $ the eigenvalues, which are
simple and positive, behave like the eigenvalues of a harmonic
oscillator. This fact allows us to obtain an amplitude function that does
not interfere with the phase.

In Section 4 we prove that every solution to $ P_{1} u = f $ is Gevrey
$ s_{0} $, if $ f \in G^{s_{0}} $. This is done using the subelliptic
estimate for the operator. 
Very likely one could have used different techniques, like
the FBI transform (see e.g. \cite{Sj-83}, \cite{Sj-Ast} and
\cite{AB}), but we thought that using the subelliptic estimate
would give the most elementary and readable proof.

Finally we collected in Appendix A the proof for a number of $
L^{\infty} $ estimates for the eigenfunctions of the Schr\"odinger
equation with a double well potential (actually they hold true in more
general cases) and in Appendix B the proof of a technical lemma that
is used in Section 4.

\bigskip
\noindent
\textbf{Acknowledgements.}
One of the authors (A. B.) would like to thank F.~Tre\-ves and
P.~Cordaro for a number of useful discussions.

\section{Statement of the Result}
\setcounter{equation}{0}
\setcounter{theorem}{0}
\setcounter{proposition}{0}
\setcounter{lemma}{0}
\setcounter{corollary}{0}
\setcounter{definition}{0}
\setcounter{remark}{0}

Let $ r $, $ p $, $ q \in \N $, $ 1 < r < p < q $, and $ x = (x_{1},
\ldots, x_{4}) \in \R^{4} $. The object of this
section is to state the optimal Gevrey regularity result for the two
operators
\begin{equation}
\label{eq:P1}
P_{1} (x, D) = D_{1}^{2} + D_{2}^{2} + x_{1}^{2(r-1)}\left(D_{3}^{2} +
  D_{4}^{2}\right) + x_{2}^{2(p-1)} D_{3}^{2} + x_{2}^{2(q-1)}
D_{4}^{2} ,
\end{equation}
\begin{equation}
\label{eq:P2}
P_{2} (x, D) = D_{1}^{2} + D_{2}^{2} + x_{1}^{2(r-1)}D_{4}^{2} +
x_{2}^{2(p-1)} D_{3}^{2} + x_{2}^{2(q-1)} D_{4}^{2} .
\end{equation}
First of all we remark that both $ P_{1} $ and $ P_{2} $ are sums of
squares of vector fields with real analytic coefficients satisfying
H\"ormander bracket condition, i.e. the whole ambient space is
generated when we take iterated commutators of the vector fields in
the definition of $ P_{j} $. In particular both $ P_{1} $ and $ P_{2}
$ are $ C^{\infty} $ hypoelliptic at the origin. This means that there
exists an open neighborhood of the origin, $ \Omega $, such that for
every open set $ V \Subset \Omega $, $ 0 \in V $, we have, for $j =1, 2$,
$$ 
P_{j} u \in C^{\infty}(V) \Rightarrow u \in C^{\infty}(V),
$$
for every distribution $ u \in \mathscr{D}'(\Omega) $.

The characteristic manifold of $ P_{1} $ is the real analytic manifold 
\begin{multline}
\label{eq:charP1}
\Char(P_{1}) = \left\{ (x, \xi) \in T^{*}\R^{4} \setminus \{0\} \ | \ 
\right .
\\  
\left .
\xi_{1} = \xi_{2} = 0, x_{1} =  x_{2} = 0, \xi_{3}^{2} + \xi_{4}^{2}
  > 0 \right\}. 
\end{multline}
The characteristic manifold of $ P_{2} $ is the real analytic variety 
\begin{multline}
\label{eq:charP2}
\Char(P_{2}) = \left\{ (x, \xi) \in T^{*}\R^{4} \setminus \{0\} \ | \ 
\right .
\\  
\left .
\xi_{1} = \xi_{2} = 0, x_{2} = 0, x_{1} \xi_{4} = 0,  
\xi_{3}^{2} + \xi_{4}^{2} > 0 \right\}. 
\end{multline}
According to Treves conjecture one has to look at the strata
associated with $ P_{1} $ and $ P_{2} $. 

The stratification associated with $ P_{1} $ is made up of a
symplectic single stratum 
$$ 
\Sigma_{1} = \left\{ (0, 0, x_{3}, x_{4}; 0, 0, \xi_{3}, \xi_{4}) \ |
\xi_{3}^{2} + \xi_{4}^{2} > 0  \ \right\} = \Char(P_{1}).
$$
The stratification associated with $ P_{2} $ is more
complicated. Actually there are eight strata of the form
\begin{itemize}
\item[a -]{} 
$$ 
\Sigma_{1, \pm} = \left\{ (0, 0, x_{3}, x_{4}; 0, 0, \xi_{3}, \xi_{4})
  \ | \ \pm \xi_{4} > 0 \right\},
$$
at depth $ r $. This is a symplectic stratum and the restriction of
the symplectic form to $\Sigma_{1, \pm}  $ has rank 4.
\item[b -]{}
$$ 
\Sigma_{2, \pm, \pm} = \left\{ (x_{1}, 0, x_{3}, x_{4}; 0, 0, \xi_{3}, 0)
  \ | \ \pm \xi_{3} > 0, \pm x_{1} > 0 \right\},
$$
at depth $ p $. This is a non symplectic stratum and the restriction of
the symplectic form to $\Sigma_{2, \pm, \pm}  $ has rank 2.
\item[c -]{}
$$ 
\Sigma_{3, \pm} = \left\{ (0, 0, x_{3}, x_{4}; 0, 0, \xi_{3}, 0)
  \ | \ \pm \xi_{3} > 0 \right\},
$$
at depth $ p $. This is a non symplectic stratum and the restriction of
the symplectic form to $\Sigma_{3, \pm}  $ has rank 2. Note that 
$$
\Sigma_{3, \pm} \subset \partial \Sigma_{2, \pm, \pm} =
\overline{\Sigma_{2, \pm, \pm}} \setminus \Sigma_{2, \pm, \pm}.
$$
\end{itemize}
According to the conjecture we would expect local real analyticity
near the origin for the distribution solutions, $ u $,  of $ P_{1}u = f $, with
a real analytic right hand side. Moreover we would also expect that $
(x, \xi) \in {}^{c}WF_{a}(u) \cap \Sigma_{1, \pm} $, for any
distribution solution of $ P_{2}u = f $, provided $
(x, \xi) \in {}^{c}WF_{a}(f) \cap \Sigma_{1, \pm} $. Here $
{}^{c}WF_{a}(u) $ denotes the complement of the analytic wave front
set of $ u $ in $ T^{*}\R^{4} \setminus \{0\} $.

We are ready to state the theorem that is proved in the next section.
\begin{theorem}
\label{th:1}
Let 
$$ 
\frac{1}{s_{0}} = \frac{1}{r} + \frac{r-1}{r} \frac{p-1}{q-1}.
$$
Then
\begin{itemize}
\item[i) ]{}
$ P_{1} $ is locally Gevrey $ s_{0} $ hypoelliptic and not better near
the origin.
\item[ii)]{}
$ P_{2} $ is microlocally Gevrey $ s_{0} $ hypoelliptic and not better
at points in $ \Sigma_{1, \pm} $.  
\end{itemize}
\end{theorem}
It is known that a sum of squares is always microlocally Gevrey $ s $
hypoelliptic for $ s \geq \sigma $, where $ \sigma $ denotes the
length of the shortest non zero Poisson bracket (see \cite{abc} for a
precise statement and a proof.) We point out that $ P_{j} $, $ j = 1,
2 $, are microlocally more regular than that.

We would like to explicitly remark that deducing the Gevrey regularity
from the geometry of the characteristic variety and/or of its
stratification seems beyond reach at this stage.

\medskip

Moreover we also have that 
\begin{theorem}
\label{th:2}
$ P_{2} $ is microlocally Gevrey $ p
$ hypoelliptic and not better at points in $ \Sigma_{2, \pm, \pm} \cup
\Sigma_{3, \pm}$.
\end{theorem}
As a consequence of Theorem \ref{th:1} we have
\begin{corollary}
\label{cor:1}
The sufficient part of Treves conjecture does not hold neither locally
nor microlocally.
\end{corollary}
We note however that for a single symplectic stratum of codimension 2
and for a single symplectic stratum of dimension 2 the conjecture is
true. For the first result we refer to \cite{ch09}. The proof of the
second is contained in \cite{bm}.  
\begin{remark}
\label{rem:rev}
We observe that in $ P_{1} $ brackets with $ D_{1} $ and brackets with
$ D_{2} $ behave differently, the latter looking closer to what
happens for the Ole\u \i nik--Radkevi\v c operator
\eqref{eq:ol-rad-op}. The more so that $ P_{1} $ is analytic
hypoelliptic if and only if $ p = q $.

Define
\begin{multline*} 
\tilde{P}_{1}(x, D) = D_{1}^{2} + D_{2}^{2} + (x_{1} + x_{2})^{2(r-1)}\left(D_{3}^{2} +
  D_{4}^{2}\right) 
\\
+ (x_{1} - x_{2})^{2(p-1)} D_{3}^{2} + (x_{1} -
x_{2})^{2(q-1)} D_{4}^{2} .
\end{multline*}
$ \tilde{P}_{1} $ has the same properties as $ P_{1} $, since it
differs from $ P_{1} $ by a rotation and a dilation. However the
brackets w.r.t. both $ D_{1} $ and $ D_{2} $ have a similar behavior:
$ \ad_{D_{j}}^{r-1} (x_{1} + x_{2})^{r-1} D_{k} = (r-1)! D_{k} $, $
j=1, 2 $, $ k = 3, 4 $.
\end{remark}

\section{Proof of Theorem \ref{th:1} and \ref{th:2}}
\setcounter{equation}{0}
\setcounter{theorem}{0}
\setcounter{proposition}{0}
\setcounter{lemma}{0}
\setcounter{corollary}{0}
\setcounter{definition}{0}
\setcounter{remark}{0}

In this section we prove the optimality of the Gevrey regularity in
Theorem \ref{th:1} as well as Theorem \ref{th:2}. 

Actually we provide a proof of statement i) of
Theorem \ref{th:1}, since the proof of statement ii), for the operator
$ P_{2} $, proceeds along the same lines and is easier.

We construct a solution to the equation $ P_{1} u = 0 $ which is not
Gevrey $ s $ for any $ s < s_{0} $ and is defined in a neighborhood of
the origin.

In fact we look for a function $ u(x, y, t) $ defined on $ \R_{x}
\times \R_{y} \times (\R^{+}_{t} \cup \{0\})$ and such that 
$$ 
P_{1}(x, D) A(u) = 0,
$$
where
\begin{equation}
\label{eq:Au}
A(u)(x) = \int_{M_{u}}^{+\infty} e^{- i \rho x_{4} + x_{3} z(\rho)
  \rho^{\theta} - \rho^{\theta}} u(\rho^{\frac{1}{r}} x_{1},
\rho^{\mu} x_{2}, \rho) d\rho,
\end{equation}
where 
$$ 
\theta = \frac{1}{s_{0}} ,
$$
$ \mu > 0 $, $ z(\rho) $ and $ M_{u}> 0 $ are to be determined. Here we assume that
$ x \in U $, a suitable neighborhood of the origin whose size will
ultimately depend on the upper estimate for $ z(\rho) $.

We have
\begin{multline*}
P_{1}(x, D) A(u)(x) \\
= \int_{M_{u}}^{+\infty} e^{- i \rho x_{4} + x_{3} z(\rho)
  \rho^{\theta} - \rho^{\theta}} \Big[ -
  \rho^{\frac{2}{r}} \partial_{x_{1}}^{2} u - x_{1}^{2(r-1)}
  (z(\rho))^{2} \rho^{2 \theta} u 
\\[5pt]
+ x_{1}^{2(r-1)} \rho^{2} u - \rho^{2\mu} \partial_{x_{2}}^{2} u -
x_{2}^{2(p-1)} (z(\rho))^{2} \rho^{2\theta} u + x_{2}^{2(q-1)}
\rho^{2} u \Big ] d\rho.
\end{multline*}
Rewriting the r.h.s. of the above relation in terms of the variables $
y_{1} = \rho^{\frac{1}{r}} x_{1} $, $ y_{2} = \rho^{\mu} x_{2} $, we
obtain
\begin{multline*}
P_{1}(x, D) A(u)(x) \\
 = \int_{M_{u}}^{+\infty} e^{- i \rho x_{4} + x_{3} z(\rho)
  \rho^{\theta} - \rho^{\theta}} \Big[ -
  \rho^{\frac{2}{r}} \partial_{1}^{2} u - y_{1}^{2(r-1)}
  (z(\rho))^{2} \rho^{2 \theta -2 \frac{r-1}{r}} u 
\\[5pt]
+ y_{1}^{2(r-1)} \rho^{2 - 2 \frac{r-1}{r}} u - \rho^{2\mu} \partial_{2}^{2} u -
y_{2}^{2(p-1)} (z(\rho))^{2} \rho^{2\theta -2(p-1)\mu} u 
\\[5pt]
+ y_{2}^{2(q-1)}
\rho^{2 - 2(q-1)\mu} u \Big ]_{\substack{y_{1} = \rho^{1/r} x_{1} \\
    y_{2} = \rho^{\mu} x_{2}}} \ d\rho.
\end{multline*}
Choose now $ \mu = \frac{1}{q} $. Then the above relation becomes
\begin{multline*}
P_{1}(x, D) A(u)(x) \\
= \int_{M_{u}}^{+\infty} e^{- i \rho x_{4} + x_{3} z(\rho)
  \rho^{\theta} - \rho^{\theta}} \Big[ -
  \rho^{\frac{2}{r}} \left( \partial_{1}^{2} - y_{1}^{2(r-1)}
  \left( 1 - (z(\rho))^{2} \rho^{2 (\theta - 1)} \right) \right) u 
\\[5pt]
+  \rho^{\frac{2}{q}} \left( - \partial_{2}^{2}  -
y_{2}^{2(p-1)} (z(\rho))^{2} \rho^{2\theta -2\frac{p}{q}}  
+ y_{2}^{2(q-1)} \right) u \Big ]_{\substack{y_{1} = \rho^{1/r} x_{1} \\
    y_{2} = \rho^{\frac{1}{q}} x_{2}}} \ d\rho.
\end{multline*}
We point out that
$$ 
\theta - 1 < 0.
$$
We make the Ansatz that 
\begin{equation}
\label{eq:Mu}
|z(\rho) | < M_{u}^{1-\theta} .
\end{equation}
We shall see that condition \eqref{eq:Mu} will be satisfied.

Set $ \tau(\rho) = \left( 1 - (z(\rho))^{2} \rho^{2 (\theta - 1)}
\right)^{\frac{1}{2r}} $. We note that, due to condition
\eqref{eq:Mu}, the quantity in parentheses is positive. Choose 
\begin{equation}
\label{eq:u}
u(y_{1}, y_{2}, \rho) = u_{1}(\tau(\rho) y_{1}) u_{2}(y_{2}, \rho),
\end{equation}
where
\begin{equation}
\label{eq:u1}
\left( - \partial_{1}^{2}  + y_{1}^{2(r-1)} \tau(\rho)^{2r} \right)
u_{1}(\tau(\rho) y_{1}) 
=  \tau(\rho)^{2} \lambda u_{1}(\tau(\rho) y_{1}) ,
\end{equation}
and $ \lambda > 0 $ is such that, for fixed $ \rho > 0 $,  the factor
in front of $ u_{1} $ in the r.h.s. of the above equation is in the
spectrum of the quantum anharmonic oscillator 
$ \left( - \partial_{1}^{2}  + y_{1}^{2(r-1)}
  \left( 1 - (z(\rho))^{2} \rho^{2 (\theta - 1)} \right) \right)$, whose
frequency depends on both $ \rho $ and $ z(\rho) $.

We then obtain
\begin{multline*}
P_{1}(x, D) A(u)(x) \\
= \int_{M_{u}}^{+\infty} e^{- i \rho x_{4} + x_{3} z(\rho)
  \rho^{\theta} - \rho^{\theta}} u_{1}(\tau(\rho) \rho^{\frac{1}{r}}x_{1}) \Big[ \Big\{
  \rho^{\frac{2}{r}} \left( 1 - (z(\rho))^{2} \rho^{2 (\theta - 1)}
  \right)^{\frac{1}{r}} \lambda 
\\[5pt]
+  \rho^{\frac{2}{q}} \left( - \partial_{2}^{2}  -
y_{2}^{2(p-1)} (z(\rho))^{2} \rho^{2\theta -2\frac{p}{q}}  
+ y_{2}^{2(q-1)} \right) \Big\} u_{2}(y_{2}, \rho) \Big ]_{y_{2} =
\rho^{\frac{1}{q}} x_{2}} \ d\rho. 
\end{multline*}
Our next step is to find $ u_{2} $ as a solution of the differential
equation
\begin{multline}
\label{eq:u2}
 \left( 1 - (z(\rho))^{2} \rho^{2 (\theta - 1)}
  \right)^{\frac{1}{r}} \lambda u
\\
+  \rho^{\frac{2}{q} - \frac{2}{r}} \left( - \partial_{2}^{2}  -
y_{2}^{2(p-1)} (z(\rho))^{2} \rho^{2\theta -2\frac{p}{q}}  
+ y_{2}^{2(q-1)} \right) \Big) u = 0 ,
\end{multline}
where we wrote $ u $ instead of $ u_{2} $ for the sake of
simplicity. \eqref{eq:u2} becomes
\begin{multline*}
 \left( 1 - (z(\rho))^{2} \rho^{2 (\theta - 1)}
  \right)^{\frac{1}{r}} \lambda u
\\
+  \rho^{\frac{2}{q} - \frac{2}{r}} \left( - \partial_{2}^{2}
+ y_{2}^{2(q-1)} \right) u  - 
(z(\rho))^{2} \rho^{2\left(\theta - \frac{p - 1}{q} -
    \frac{1}{r}\right)} y_{2}^{2(p-1)} u = 0 ,
\end{multline*}
Since
$$ 
\theta - \frac{p - 1}{q} - \frac{1}{r} = \left( \frac{1}{q} -
  \frac{1}{r}\right) \frac{p-1}{q-1},
$$
we set
\begin{equation}
\label{eq:t}
t = \rho^{\frac{1}{q} - \frac{1}{r}} ,
\end{equation}
so that the above equation becomes
\begin{multline*}
 \left( 1 - (z_{1}(t))^{2} t^{2 (r-1)\frac{q}{q-1}\ \frac{q-p}{q-r} }
  \right)^{\frac{1}{r}} \lambda u
\\
+  t^{2} \left( - \partial_{2}^{2}
+ y_{2}^{2(q-1)} \right) u  - 
(z_{1}(t))^{2} t^{2 \frac{p - 1}{q-1}} y_{2}^{2(p-1)} u = 0 ,
\end{multline*}
where $ z_{1}(t) = z(\rho) $.
The latter equation can be turned into a stationary semiclassical
Schr\"odinger equation if we perform the canonical dilation 
$$ 
y_{2} = y t^{- \frac{1}{q-1}} :
$$
\begin{multline*}
 \left( 1 - (z_{1}(t))^{2} t^{2 (r-1)\frac{q}{q-1}\ \frac{q-p}{q-r} }
  \right)^{\frac{1}{r}} \lambda u
\\
- t^{2 \frac{q}{q-1}} \partial_{y}^{2} u
+ y^{2(q-1)}  u  - 
(z_{1}(t))^{2} y^{2(p-1)} u = 0 .
\end{multline*}
Set
\begin{equation}
\label{eq:h}
h = t^{\frac{q}{q-1}}.
\end{equation}
Note that $ t $, $ h $ are small and positive for large $ \rho $. Thus
we may rewrite the above equation as
\begin{equation}
\label{eq:eqh}
\Big [\left( 1 - (z_{2}(h))^{2} h^{2 (r-1) \frac{q-p}{q-r} }
  \right)^{\frac{1}{r}} \lambda 
- h^2 \partial_{y}^{2} 
+ y^{2(q-1)}    - 
(z_{2}(h))^{2} y^{2(p-1)} \Big] u = 0 ,
\end{equation}
where $ z_{2}(h) = z_{1}(t) $.

We want to show that there are countably many choices for the function
$ z_{2}(h) $ in such a way that equation \eqref{eq:eqh} has a non zero
solution in $ L^{2}(\R) $, which actually turns out to be a smooth
rapidly decreasing function. 

Since the term containing $ \lambda $ is a scalar quantity commuting
with the other terms, we consider first the operator
$$ 
- h^2 \partial_{y}^{2} + y^{2(q-1)} - (z_{2}(h))^{2} y^{2(p-1)}.
$$
This is a Schr\"odinger operator with a symmetric double well
potential. The latter is not positive everywhere; in order to work
with a positive double well potential we subtract (and add) its
minimum. This is
$$ 
\hat{\gamma} z_{2}^{2 \frac{q-1}{q-p}},
$$
where
$$ 
\hat{\gamma} = - \frac{q-p}{q-1} \left(
  \frac{p-1}{q-1}\right)^{\frac{p-1}{q-p}} < 0.
$$
Equation \eqref{eq:eqh} becomes
\begin{multline}
\label{eq:eqh0}
\Big [\left( 1 - (z_{2}(h))^{2} h^{2 (r-1) \frac{q-p}{q-r} }
  \right)^{\frac{1}{r}} \lambda + \hat{\gamma} z_{2}(h)^{2 \frac{q-1}{q-p}} 
\\
- h^2 \partial_{y}^{2} 
+ y^{2(q-1)}    - 
(z_{2}(h))^{2} y^{2(p-1)} - \hat{\gamma} z_{2}(h)^{2 \frac{q-1}{q-p}} \Big] u = 0 ,
\end{multline}
Let us make the Ansatz that $ z_{2} $ is a positive valued
function. We make the canonical dilation
$$ 
y_{2} = x z_{2}^{\frac{1}{q-p}}.
$$
Then \eqref{eq:eqh0} becomes
\begin{multline}
\label{eq:eqh1}
\Big [\left( 1 - (z_{2}(h))^{2} h^{2 (r-1) \frac{q-p}{q-r} }
  \right)^{\frac{1}{r}} \lambda + \hat{\gamma} z_{2}(h)^{2 \frac{q-1}{q-p}} 
- h^2 z_{2}(h)^{-\frac{2}{q-p}} \partial_{x}^{2} 
\\
+ z_{2}(h)^{2 \frac{q-1}{q-p}} x^{2(q-1)}    - 
(z_{2}(h))^{2\frac{q-1}{q-p}} x^{2(p-1)} -
\hat{\gamma} z_{2}(h)^{2 \frac{q-1}{q-p}} \Big] u = 0 , 
\end{multline}
whence
\begin{multline}
\label{eq:eqh2}
\Big [\left( 1 - (z_{2}(h))^{2} h^{2 (r-1) \frac{q-p}{q-r} }
  \right)^{\frac{1}{r}} z_{2}(h)^{-2\frac{q-1}{q-p}} \lambda + \hat{\gamma}
\\
- h^2 z_{2}(h)^{-\frac{2q}{q-p}} \partial_{x}^{2} 
+ x^{2(q-1)} - x^{2(p-1)} -
\hat{\gamma} \Big] u = 0 , 
\end{multline}
Let us consider the one dimensional Schr\"odinger  operator
\begin{equation}
\label{eq:buca}
- \left(h z_{2}(h)^{-\frac{q}{q-p}}\right)^{2} \partial_{x}^{2} +
x^{2(q-1)} - x^{2(p-1)} - \hat{\gamma} .
\end{equation}
By \cite{berezin} (Chapter 2, Theorem 3.1) it has a discrete simple
spectrum depending in a real analytic way on the parameter $ h
z_{2}(h)^{-\frac{q}{q-p}} $, for $ h > 0 $. Let us denote by
$$ 
E\left( \frac{h}{z_{2}(h)^{\frac{q}{q-p}}} \right)
$$
an eigenvalue. Let $ u = u(x, h) $ be the corresponding
eigenfunction. Then \eqref{eq:eqh2} becomes
\begin{equation}
\label{eq:eqh3}
\left( 1 - (z_{2}(h))^{2} h^{2 (r-1) \frac{q-p}{q-r} }
  \right)^{\frac{1}{r}} z_{2}(h)^{-2\frac{q-1}{q-p}} \lambda +
  \hat{\gamma} + E\left( \frac{h}{z_{2}(h)^{\frac{q}{q-p}}} \right) = 0.
\end{equation}
Next we are going to solve the above equation w.r.t. $ z_{2} $ as a
function of $ h $, for small positive values of $ h $. 
\begin{proposition}
\label{prop:z2}
There is $ h_{0} > 0 $ such that equation \eqref{eq:eqh3} implicitly
defines a function $ z_{2} \in C([0, h_{0}[) \cap C^{\omega}(]0,
h_{0}[) $. In particular 
$$ 
z_{2}(h) \to \tilde{z} = \left( -
  \frac{\lambda}{\hat{\gamma}}\right)^{\frac{q-p}{2(q-1)}} > 0
$$
when $ h \to 0+ $. Therefore we may always assume that
\begin{equation}
\label{eq:z2range}
z_{2}(h) \in \left[ \frac{1}{2} \tilde{z} , \frac{3}{2} \tilde{z} \right],
\end{equation}
for $h \in [0, h_{0}[  $.
\end{proposition}
\begin{proof}
The operator in \eqref{eq:buca} has a symmetric non negative double
well potential with two non degenerate minima and unbounded at
infinity. From Theorem 1.1 in \cite{simon_83} we deduce that
\begin{equation}
\label{eq:E}
\lim_{\mu \to 0+} \frac{E(\mu)}{\mu} = e^{*} > 0.
\end{equation}
We may then continue the function $ E $, by setting $ E(0) = 0 $, as a
function in $ C([0, +\infty[) \cap C^{\omega}(]0, +\infty[) $. 

Set
\begin{equation}
\label{eq:f}
f(h, z) = \left( 1 - z^{2} h^{2 (r-1) \frac{q-p}{q-r} }
  \right)^{\frac{1}{r}} z^{-2\frac{q-1}{q-p}} \lambda +
  \hat{\gamma} + E\left( \frac{h}{z^{\frac{q}{q-p}}} \right) .
\end{equation}
Note that $ f(0, \tilde{z}) = 0 $. We want to show that the equation $
f(h, z) = 0 $ can be uniquely solved w.r.t. $ z $ for $ h \in [0,
h_{0}] $, for a suitable $ h_{0} $. 

We need the
\begin{lemma}
\label{lemma:dE}
For every $ \mu_{0} > 0 $ we have that $ \partial_{\mu} E(\mu) $ exists
and is bounded for $ 0 \leq \mu \leq \mu_{0} $.
\end{lemma}
\begin{proof}[Proof of Lemma \ref{lemma:dE}]
Let 
$$ 
Q_{\mu}(x, \partial_{x}) = - \mu^{2} \partial_{x}^{2} +
x^{2(q-1)} - x^{2(p-1)} - \hat{\gamma} .
$$
From $ Q_{\mu} v = E(\mu) v $ we get 
$$ 
\langle Q_{\mu} \partial_{\mu}v , v\rangle + 2 \mu \| \partial_{x}
v\|^{2} = E(\mu) \langle  \partial_{\mu}v , v\rangle + (\partial_{\mu}
E(\mu)) \| v\|^{2}.
$$
Due to the self adjointness of $ Q_{\mu} $ the first terms on both
sides of the above identity are equal, so that
$$ 
\partial_{\mu} E(\mu) = 2 \mu \| \partial_{x} v\|^{2} \geq 0,
$$
for every $ \mu > 0 $, provided $ v $ is normalized. Again from $
Q_{\mu} v = E(\mu) v $ we deduce that 
\begin{equation}
\label{eq:v'} 
\mu^{2} \| \partial_{x}
v\|^{2} \leq \langle Q_{\mu} v , v \rangle = E(\mu). 
\end{equation}
Hence
$$ 
0 \leq \partial_{\mu} E(\mu) \leq 2 \frac{E(\mu)}{\mu} \to 2 e^{*}
$$
for $ \mu \to 0+ $. The existence of the right derivative in $ \mu = 0
$ is a consequence of \eqref{eq:E}.
\end{proof}

Now
\begin{multline*}
\frac{\partial f}{\partial z}(h, z) = - \frac{2}{r} \lambda 
\left( 1 - z^{2} h^{2 (r-1) \frac{q-p}{q-r} } \right)^{\frac{1 -
    r}{r}} z^{1-2\frac{q-1}{q-p}} h^{2 (r-1) \frac{q-p}{q-r} } 
\\
- 2 \frac{q-1}{q-p} \left( 1 - z^{2} h^{2 (r-1) \frac{q-p}{q-r} }
  \right)^{\frac{1}{r}} z^{-2\frac{q-1}{q-p} - 1} \lambda 
- \frac{q}{q-p} E'\left( \frac{h}{z^{\frac{q}{q-p}}} \right)
\frac{h}{z^{\frac{q}{q-p}}} z^{-1} .
\end{multline*}
The above quantity is strictly negative if $ (h, z) \in [0, h_{0}[
\times [\tilde{z} - \delta, \tilde{z} + \delta ] $, for a suitable
choice of small $ h_{0} $, $ \delta $. 

Because of the definition of $ \tilde{z} $ and \eqref{eq:f}, $ f(h,
\tilde{z}-\delta) > 0 $, $ f(h, \tilde{z}+\delta) < 0  $ possibly
taking a smaller $ h_{0} $, $ \delta $, for $ 0 \leq h \leq h_{0}
$. Since $ f $ is continuous and strictly decreasing on the $ h
$-lines there is a unique zero of the equation $ f(h, z(h)) = 0 $ with
$ z(h) \in [\tilde{z} - \delta, \tilde{z} + \delta ] $ for $ 0 \leq h
\leq h_{0} $. 

For positive $ h $ trivially $ z(h) $ is real analytic. Let us show
that $z(h) \in C([0, h_{0}[) $. Arguing by contradiction assume that $
z(h) \not \to \tilde{z} $ for $ h \to 0+ $. Then there is a sequence $
h_{k} \to 0+ $ such that $ z(h_{k}) \to \hat{z} \not = \tilde{z}
$. Then $ 0 = f(h_{k}, z(h_{k})) \to f(0, \hat{z})  $ which is false
since $ \tilde{z} $ is the only zero of $ f(0, z) = 0 $.
\end{proof}
\begin{remark}
\label{rem:Mu}
Let $ h_{0} $ be the quantity define in Proposition \ref{prop:z2}. 
Set $ h_{0} = \rho_{0}^{\left(\frac{1}{q} - \frac{1}{r}\right)
  \frac{q}{q-1}}  $. 
Choosing $ M_{u} \geq \max \{ \rho_{0} ,
(\frac{3}{2}\tilde{z})^{\frac{1}{1-\theta}}\} $ we have that the function $
z_{2} $ is defined for $ \rho \geq M_{u} $ and that \eqref{eq:Mu} is
satisfied, so that $ 1 - z(\rho)^{2} \rho^{2(\theta-1)} > 0 $.
\end{remark}

We are going to need also a couple of lemmas whose proofs can be found
in Appendix A.
\begin{lemma}
\label{lemma:apriori}
There exists $ \mu_{0} > 0 $ such that for $ v \in \mathscr{S}(\R) $ the
following a priori inequality holds
\begin{equation}
\label{eq:apriori}
\mu^{2} \| v'' \| + \| V v \| \leq C \left( \| Q_{\mu} v \| + \mu \| v
  \| \right),
\end{equation}
for a positive constant $ C $ independent of $ \mu \in ]0, \mu_{0}[ $. 
\end{lemma}
\begin{lemma}
\label{lemma:vk}
Let  $ v(x, \mu) $ denote the $ L^{2}(\R) $ normalized eigenfunction corresponding
to $ E(\mu) $. $ v $  is rapidly
decreasing w.r.t. $ x $ and satisfies the estimates
\begin{equation}
\label{eq:vj}
| v^{(j)} (x, \mu) | \leq C_{j} \mu^{-(j+1)/2},
\end{equation}
for $ x \in \R $, $ C_{j} > 0 $ independent of  $0 < \mu < \mu_{0} $, $
j = 0, 1, 2 $, $ \mu_{0} $ suitably small. 
\end{lemma}
\begin{remark}
\label{rem:1}
Note that because of Lemma \ref{lemma:vk} the formal method of sliding
the differential operator $ P_{1} $ under the integration sign becomes
legitimate, since a power singularity does no harm to the convergence
of the integral in $ \rho $.
\end{remark}
The integral $ A(u) $ in \eqref{eq:Au} is a convergent integral since
$ u = u_{1} u_{2} $, $ u_{i} $ is a real analytic function of $ x_{i}
$, $ u_{1} $ is rapidly decreasing, while $ u_{2}
$ is rapidly decreasing w.r.t. $ x $ and satisfies \eqref{eq:vj} with $
\mu = \mathscr{O}\left(\rho^{\left(\frac{1}{q} - \frac{1}{r}\right)
  \frac{q}{q-1}}\right) $.

Then \eqref{eq:eqh2} holds for $ 0 \leq h \leq h_{0} $ and hence for $
\rho \geq \rho_{0} $.

Going back to \eqref{eq:Au} we see that
$$ 
P_{1}(x, D) A(u) = 0.
$$
Before concluding the proof of the sharpness of the Gevrey $ s_{0} $
regularity for $ A(u) $, we need to make sure that the function $ u =
u_{1} u_{2} $ does not have any effect on the convergence of the
integral at infinity as well as on the Gevrey behavior of $
A(u) $. 

As far as $ u_{1} $ is concerned, this is fairly obvious, since
$ u_{1} $ is a rapidly decreasing function of $ \tau(\rho) \rho^{\frac{1}{r}}
x_{1} $, where $ \tau(\rho) $ is defined before equation \eqref{eq:u},
and, computing this function at the origin---as we need to 
do---will not affect the exponential in $ A(u) $. We are thus left
with $ u_{2} = u_{2}(\rho^{\frac{1}{q}} x_{2}, \rho) $. Even though $
u_{2} $ is rapidly decreasing w.r.t. $ \rho^{\frac{1}{q}} x_{2} $, we
still need some estimate on $ u_{2} $ allowing us to conclude that $
u_{2} $ can be polynomially bounded in $ \rho $, uniformly for $ x_{2}
$ in a neighborhood of the origin and moreover that $ u_{2}(0, \rho) $
does not vanish with so high a speed to compromise the Gevrey $ s_{0}
$ regularity. 

Let us then consider $ u_{2} = u_{2}(x, \hslash) $. It is an
eigenfunction of the operator
\begin{equation}
\label{eq:Qh}
Q_{\hslash}(x, \partial_{x}) = - \hslash^{2} \partial_{x}^{2} +
x^{2(q-1)} - x^{2(p-1)} - \hat{\gamma} ,
\end{equation}
where, by \eqref{eq:eqh2},
\begin{equation}
\label{eq:hslash}
\hslash = \frac{h}{z_{2}(h)^{\frac{q}{q-p}}}.
\end{equation}
Note that $ \hslash $ tends to zero if and only if $ h $ tends to
zero. 
For the sake of simplicity we write $ u $ instead of $ u_{2} $; $ u $
is an eigenfunction of \eqref{eq:Qh} corresponding to the eigenvalue $
E(\hslash) $
\begin{equation}
\label{eq:eigen}
Q_{\hslash}(x, \partial_{x}) u = E(\hslash) u.
\end{equation}
Obviously $ u $ is a solution of \eqref{eq:eqh2}.

We explicitly observe that the choice of both $ u_{1} $ and $ u_{2} $
is arbitrary among the eigenfunctions of the corresponding
Schr\"odinger operators in the variables $ x_{1} $ and $ x_{2} $. It
would have probably been easier to deal with the first eigenfunction of
both operators, which can be studied with elementary me\-thods, but we
prefer to emphasize the fact that all the eigenfunctions share the
same properties that are needed to accomplish our construction. 

We are going to need an estimate of $ u $ in a classically forbidden
region, i.e. when $ \hslash $ is small ($ \rho $ is large) and $ x $
is in a neighborhood of the origin. We recall the following theorem
providing a lower bound for the tunneling of the solution:
\begin{theorem}[See \cite{zworski}, Theorem 7.7]
\label{th:tunnel}
Let $ U $ be a neighborhood of the origin in $ \R $. There exist
positive constants $ C $, $ \hslash_{0} $ such that
\begin{equation}
\label{eq:esttunnel}
\| u \|_{L^{2}(U)} \geq e^{- \frac{C}{\hslash}} \| u \|_{L^{2}(\R)} ,
\end{equation}
for $ 0 < \hslash \leq \hslash_{0} $. 
\end{theorem}
The Schr\"odinger operator $ Q_{\hslash} $ has a symmetric potential,
so that its eigenfunctions are either even or odd functions w.r.t. the
variable $ x $. 

\medskip
\noindent
\textbf{Case of an even eigenfunction.} We may assume that 
$$ 
\| u \|_{L^{2}(\R)} =1, \quad u(0, \hslash) > 0,
$$
since $ u'(0, \hslash) = 0 $ because of its parity and if $ u(0,
\hslash) = 0 $ would imply that $ u $, being a solution of
\eqref{eq:eigen}, is identically zero.

Moreover, by \eqref{eq:eigen}, $ \partial_{x}^{2}u(0, \hslash) > 0 $. 

Denote by $ x_{0} = x_{0}(\hslash) $ the first positive zero of $ V(x)
- E(\hslash) = x^{2(q-1)} - x^{2(p-1)} - \hat{\gamma} - E(\hslash)$. 
Note that $ u $ is strictly positive in the interval $ 0\leq x \leq
x_{0} $. In fact, by contradiction, denoting by $ \bar{x} $ the first
zero of $ u $ in $ [0, x_{0}] $, by \eqref{eq:eigen}, we may conclude
that $ u'' > 0 $ in $ [0, \bar{x}[ $ so that the same is true for $ u'$. 
Hence $ u(\bar{x}, \hslash) > u(0, \hslash) > 0 $, which is absurd.

By \eqref{eq:eigen}, $ u $ is strictly convex for $ 0 \leq x \leq x_{0} $ and has its minimum
at the origin and its maximum at $ x_{0} $. 

Define $ y = \frac{\partial_{x}u}{u} $. We have $ y > 0 $ if $ 0 < x
\leq x_{0} $. Then, writing $ y' $ for $ \partial_{x}y $,
$$ 
y' = \frac{V - E}{\hslash^{2}} - y^{2}.
$$
The function $ y $ has a maximum in the interval $ ]0, x_{0}[ $. In
fact $ y'(0) > 0 $ and $ y'(x_{0}) = - y^{2}(x_{0}) < 0 $. 
Denote by $ \bar{x} $ the point where where the maximum is attained:
it lies in the interior of the interval $ [0, x_{0}] $. 
Moreover we get 
$$ 
y(\bar{x}) = \frac{ \left(V(\bar{x}) - E(\hslash)\right)^{1/2}}{\hslash}.  
$$
From the definition of $ y $ we obtain
\begin{multline*}
u(0) = e^{- \int_{0}^{x_{0}} y(s) ds } u(x_{0}) 
\geq
e^{- x_{0} y(\bar{x})} \frac{1}{\sqrt{2 x_{0}}} \| u
\|_{L^{2}([-x_{0}, x_{0}])}
\\
\geq
\frac{1}{\sqrt{2 x_{0}}}
e^{- \frac{(-\hat{\gamma})^{1/2}}{\hslash}}
 e^{- \frac{C}{\hslash}}.
\end{multline*}
Here we used Theorem \ref{th:tunnel}, $ x_{0} < 1 $, $ E(\hslash) > 0
$ and $ u $ is normalized. We remark that $ x_{0}(\hslash) \to
\hat{x}_{0} > 0 $ when $ \hslash \to 0+ $. 

\medskip
We are now in a position to conclude the proof of Theorem \ref{th:1}
for an even function $ u_{2} $. We recall that
$$ 
\hslash = \mathscr{O}\left(\rho^{\left(\frac{1}{q} - \frac{1}{r}\right)
  \frac{q}{q-1}}\right) = \mathscr{O}\left(\rho^{-\kappa}\right).
$$
We compute
\begin{multline*}
(-D_{x_{4}})^{k} \partial_{x_{1}}^{\epsilon} A(u)(0) = 
\int_{M_{u}}^{+\infty} e^{- \rho^{\theta}} \rho^{k + \frac{\epsilon}{r}}
\tau(\rho)^{\epsilon} \partial^{\epsilon}u_{1}(0)
u_{2}(0, \rho)  d\rho
\\
\geq
\partial^{\epsilon}u_{1}(0) C \int_{M_{u}}^{+\infty} e^{- \rho^{\theta} -
  C_{1} \rho^{\kappa}} \tau(\rho)^{\epsilon}
\rho^{k + \frac{\epsilon}{r}}  d\rho
\geq
C_{2}^{k+1} k!^{s_{0}},
\end{multline*}
where $ \epsilon = 0 $ or $ 1 $ if $ u_{1} $ is even or odd
respectively and 
$$ 
\kappa = \left(\frac{1}{r} - \frac{1}{q}\right)\frac{q}{q-1} < \theta.
$$
The last inequality above holds since
\begin{multline*}
\int_{M_{u}}^{+\infty} e^{- \rho^{\theta} -
  C_{1} \rho^{\kappa}} \tau(\rho)^{\epsilon}
\rho^{k + \frac{\epsilon}{r}}  d\rho
\geq 
C_{\tau} \int_{M_{u}}^{+\infty} e^{- c \rho^{\theta}} 
\rho^{k}  d\rho 
\\
=
-C_{\tau} \int_{0}^{M_{u}} e^{- c \rho^{\theta}} \rho^{k}  d\rho +
C^{k+1}_{2} k!^{s_{0}} 
\\
\geq C_{2}^{k+1} k!^{s_{0}} \left( 1 - C_{\tau} C_{2}^{-(k+1)} M_{u} e^{- c M_{u}^{\theta}}
  \frac{M_{u}^{k}}{k!^{s_{0}}} \right)
\geq
C_{3}^{k+1} k!^{s_{0}},
\end{multline*}
if $ k $ is suitably large and $ C_{3} $ is suitable.

\medskip
\noindent
\textbf{Case of an odd eigenfunction.}
We may assume that 
$$ 
\| u \|_{L^{2}(\R)} =1, \quad u'(0, \hslash) > 0.
$$
Moreover, due to the parity $ u''(0, \hslash) = 0
$. Arguing as above we obtain that $ u' $ is positive in $ [0, x_{0}]$.  
Set
$$ 
y = \frac{u''}{u'}. 
$$
Arguing as above we deduce 
$$
u'(0, \hslash) \geq e^{- \frac{(- \hat{\gamma})^{1/2}}{\hslash}}
u'(x_{0}, \hslash) \geq \frac{1}{\sqrt{2 x_{0}}} e^{- \frac{(-
    \hat{\gamma})^{1/2}}{\hslash}} \| u' \|_{L^{2}([-x_{0}, x_{0}])}.
$$
Since
$$ 
\| u \|_{L^{2}([-x_{0}, x_{0}])} \leq x_{0}  \| u' \|_{L^{2}([-x_{0},
  x_{0}])}, 
$$
we get 
$$ 
u'(0, \hslash) \geq \frac{1}{x_{0} \sqrt{2 x_{0}}} e^{- \frac{(-
    \hat{\gamma})^{1/2}}{\hslash}} \| u \|_{L^{2}([-x_{0}, x_{0}])} .
$$
Using Theorem \ref{th:tunnel} as before we can conclude exactly as in
the case of an even eigenfunction.

\medskip
To finish the proof of Theorem \ref{th:1} we recall that Lemma \ref{lemma:vk} 
implies that the integral in the definition of $
A(\partial_{x_{2}} u) $ is absolutely convergent, so that, arguing as
before, we have
\begin{multline*}
(-D_{x_{4}})^{k} \partial_{x_{1}}^{\epsilon} A(\partial_{x_{2}} u)(0) 
\\
= 
\int_{M_{u}}^{+\infty} e^{- \rho^{\theta}} \rho^{k +
  \frac{\epsilon}{r} + \frac{1}{q}} 
\tau(\rho)^{\epsilon} \partial^{\epsilon}u_{1}(0)
\partial_{x_{2}} u_{2}(0, \rho)  d\rho
\\
\geq
\partial^{\epsilon}u_{1}(0) C \int_{M_{u}}^{+\infty} e^{- \rho^{\theta} -
  C_{1} \rho^{\kappa}}
\rho^{k + \frac{\epsilon}{r}+\frac{1}{q}}  d\rho
\geq
C_{2}^{k+1}  k!^{s_{0}},
\end{multline*}
again provided $ k $ is suitably large.

This concludes the proof of Theorem \ref{th:1}.

\bigskip
\begin{proof}[Proof of Theorem \ref{th:2}]
Optimality: It is enough to use a solution independent of $ x_{4}
$. This solves a generalized Baouendi-Goulaouic operator and it is
well known that one cannot have a Gevrey regularity better than $ p
$. 

Sufficiency: By \cite{abc}, $ P_{2} $ is microlocally Gevrey $ p $
hypoelliptic on the strata of depth $ p $.
\end{proof}

\section{Proof of Theorem \ref{th:1} (Cont'd)}
\setcounter{equation}{0}
\setcounter{theorem}{0}
\setcounter{proposition}{0}
\setcounter{lemma}{0}
\setcounter{corollary}{0}
\setcounter{definition}{0}
\setcounter{remark}{0}

In this section we prove that the operator $ P_{1} $ is Gevrey $ s_{0}
$ hypoelliptic. It is useful to establish a notation for the vector
fields defining $ P_{1} $:
\begin{eqnarray}
\label{eq:1bis}
P_{1} (x, D) & = & D_{1}^{2} + D_{2}^{2} + x_{1}^{2(r-1)}D_{3}^{2} +
 x_{1}^{2(r-1)} D_{4}^{2}  
\\
&  &
+ x_{2}^{2(p-1)} D_{3}^{2} + x_{2}^{2(q-1)} D_{4}^{2}  \notag
\\
& = & \sum_{j =1}^{6} X_{j}(x, D)^{2} 
\notag
\end{eqnarray}
We note that, using commutators of the fields up to the length $ r $,
we generate the ambient space.

The basic idea for the proof is to use the subelliptic estimate (see
e. g. \cite{hormander67} and \cite{BT} for the method of proof)
\begin{equation}
\label{eq:subell}
\| u \|_{\frac{1}{r}}^{2} + \sum_{j=1}^{6} \| X_{j}(x, D) u \|^{2} \leq C
\left( \langle P_{1}(x, D) u, u \rangle + \| u \|^{2} \right),
\end{equation}
where $ C $ is a positive constant, $ \| \cdot \|_{\frac{1}{r}} $ is
the Sobolev norm of order $ \frac{1}{r} $ and $ u \in
C_{0}^{\infty}(\R^{4}) $.   

A further remark is that we may assume $ \xi_{4} \geq 1 $: in fact
denoting by $ \psi $ a cutoff function such that $ \psi \geq 0 $, $
\psi(\xi_{4}) = 1 $ if $ \xi_{4} \geq 2 $ and $ \psi(\xi_{4}) = 0 $ if
$ \xi_{4} \leq 1 $, we may apply $ \psi(D_{4}) $ to the equation $
P_{1} u = f $ getting $ P_{1} \psi u = \psi f $, since $ \psi $
commutes with $ P_{1} $. On the other hand $ \psi f \in G^{s} $
if $ f \in G^{s} $, for $ s \geq 1 $, and we are interested in the
microlocal Gevrey regularity of $ u $ at the point $ (0; e_{4}) $. We
write $ u $ instead of $ \psi u $. 

Let us denote by $ \phi_{N} = \phi_{N}(x_{3}, x_{4})$, $ \chi_{N} =
\chi_{N}(\xi_{4}) $ cutoffs of the type defined in Appendix B. We want
to estimate the quantity $ \| X_{j} \phi_{N} D_{4}^{N} u \| $, $ j =
1, \ldots, 6 $, so that getting an estimate of the form
$ \| X_{j} \phi_{N} D_{4}^{N} u \| \leq C^{N+1} N^{s_{0} N} $ will be
enough to conclude that $ u \in G^{s_{0}} $ microlocally at $ (0,
e_{4}) $.  

As a preliminary remark we point out that if $ P_{1} u = f $, $ f \in
G^{s_{0}}(\Omega) $, then we may assume that $ u \in
C^{\infty}(\Omega) $ and has compact support w.r.t. the variables $
x_{1} $, $ x_{2} $. In fact, if $ \theta = \theta(x_{1}, x_{2}) \in
G^{s_{0}} \cap C_{0}^{\infty} $ and is identically equal to 1 in a
neighborhood of the origin, we obtain, multiplying the equation $
P_{1} u = f $ by $ \theta $, that $ P_{1}(\theta u) = \theta f -
[P_{1}, \theta] u $. We write $ u $ instead of $ \theta u $.

Now
\begin{multline}
\label{eq:chi}
\| X_{j} \phi_{N} D_{4}^{N} u \| 
\leq \| X_{j} \phi_{N} ( 1 - \chi_{N}(N^{-1} D_{4})) D_{4}^{N} u \| 
\\
+ \| X_{j} \phi_{N}
\chi_{N}(N^{-1} D_{4}) D_{4}^{N} u \|.
\end{multline}
Consider the first summand above. Since $ (1 - \chi_{N}) \psi $ has
support for $ 1 \leq \xi_{4} \leq N $, we deduce immediately a bound
of the first summand:
$$ 
\| X_{j} \phi_{N} ( 1 - \chi_{N}(N^{-1} D_{4})) D_{4}^{N} u \| \leq
C N^{N} ,
$$
where $ C $ denotes a positive constant independent of $ N $, but
depending on $ u $. This means a real analytic growth rate for $ u
$. It is enough then to bound the second summand in \eqref{eq:chi}.

To do this we plug the quantity $ \phi_{N} \chi_{N} D_{4}^{N} u $ into
\eqref{eq:subell} and, as a consequence, we obtain
\begin{multline*}
\| X_{j} \phi_{N} \chi_{N}(N^{-1} D_{4}) D_{4}^{N} u \|^{2} \leq
\| \phi_{N} \chi_{N}(N^{-1} D_{4}) D_{4}^{N} u \|_{\frac{1}{r}}^{2} 
\\
+
\sum_{j=1}^{6} \| X_{j}(x, D) \phi_{N} \chi_{N}(N^{-1} D_{4})
D_{4}^{N} u \|^{2} 
\\
\leq C \left(
\langle P_{1} \phi_{N} \chi_{N}(N^{-1} D_{4}) D_{4}^{N} u , \phi_{N}
\chi_{N}(N^{-1} D_{4}) D_{4}^{N} u \rangle 
\right .
\\
\left . \vphantom{D^{N}}
+
\| \phi_{N} \chi_{N}(N^{-1} D_{4}) D_{4}^{N} u \|^{2}
\right)
\end{multline*}
Our main concern is the estimate of the scalar product in the next to
last line of the above formula.  
We have
\begin{multline*}
\langle P_{1} \phi_{N} \chi_{N}(N^{-1} D_{4}) D_{4}^{N} u , \phi_{N}
\chi_{N}(N^{-1} D_{4}) D_{4}^{N} u \rangle 
\\
=
\langle \phi_{N} \chi_{N}(N^{-1} D_{4}) D_{4}^{N} P_{1} u , \phi_{N}
\chi_{N}(N^{-1} D_{4}) D_{4}^{N} u \rangle 
\\
+
\sum_{j=1}^{6}
\langle [ X_{j}^{2} , \phi_{N} ]  \chi_{N}(N^{-1} D_{4}) D_{4}^{N} u , \phi_{N}
\chi_{N}(N^{-1} D_{4}) D_{4}^{N} u \rangle .
\end{multline*}
The first term in the r.h.s. of the above relation poses no problem:
in fact $ P_{1} u = f \in G^{s_{0}} $ and thus the scalar product is
easily estimated by $ C^{N+1} N!^{s_{0}} $, while the right factor can
be absorbed on the left.

As for the summands containing a commutator, 
\begin{multline}
\label{eq:12comm}
\langle [ X_{j}^{2} , \phi_{N} ]  \chi_{N}(N^{-1} D_{4}) D_{4}^{N} u , \phi_{N}
\chi_{N}(N^{-1} D_{4}) D_{4}^{N} u \rangle 
\\
=
2 \langle [ X_{j} , \phi_{N} ]  \chi_{N}(N^{-1} D_{4}) D_{4}^{N} u ,
X_{j} \phi_{N} \chi_{N}(N^{-1} D_{4}) D_{4}^{N} u \rangle 
\\
-
\langle  N^{-1} [ X_{j}  [ X_{j}  , \phi_{N} ] ]  \chi_{N}(N^{-1}
D_{4}) D_{4}^{N} u , N \phi_{N} \chi_{N}(N^{-1} D_{4}) D_{4}^{N} u
\rangle .  
\end{multline}
Here we multiplied and divided by  $ N $ the factors of the second
scalar product to compensate for the second derivative landing on $
\phi_{N} $ because of the double commutator. The na\"\i ve idea behind
this is that one derivative of $ \phi_{N} $ is worth $ N $.

We are going to examine the terms with a single commutator first.
Both $ X_{1} $, $ X_{2} $ commute with $ \phi_{N} $ at this moment,
even though we shall see shortly that this is not going to be true any
longer for $ X_{2} $, at least in certain cases. Moreover
$  [ X_{j} , \phi_{N} ]  \chi_{N}(N^{-1} D_{4}) D_{4}^{N} u =
x_{1}^{r-1} \phi_{N}'  \chi_{N}(N^{-1} D_{4}) D_{4}^{N} u $, for $ j =
3, 4$. Here we just denote by $ \phi_{N}' $ a (self-adjoint)
derivative w.r.t. $ x_{3} $ or $ x_{4} $, since a more precise
notation would only burden the exposition. For $ j=5 $ we have
$  [ X_{j} , \phi_{N} ]  \chi_{N}(N^{-1} D_{4}) D_{4}^{N} u =
x_{2}^{p-1} \phi_{N}'  \chi_{N}(N^{-1} D_{4}) D_{4}^{N} u $ and for $
j = 6 $, $  [ X_{j} , \phi_{N} ]  \chi_{N}(N^{-1} D_{4}) D_{4}^{N} u =
x_{2}^{q-1} \phi_{N}'  \chi_{N}(N^{-1} D_{4}) D_{4}^{N} u $.

Let us consider the terms corresponding to $ j = 3, 4 $ first.
\begin{multline*}
2 \left |\langle  x_{1}^{r-1} \phi_{N}'  \chi_{N}(N^{-1} D_{4}) D_{4}^{N} u ,
X_{j} \phi_{N} \chi_{N}(N^{-1} D_{4}) D_{4}^{N} u \rangle \right| 
\\
\leq
\delta \| X_{j} \phi_{N} \chi_{N}(N^{-1} D_{4}) D_{4}^{N} u \|^{2} +
\frac{1}{\delta} \| x_{1}^{r-1} \phi_{N}'  \chi_{N}(N^{-1} D_{4})
D_{4}^{N} u \|^{2} ,
\end{multline*}
where $ \delta $ is a positive number so small to allow us to absorb
the first summand in the r.h.s. above on the left of the subelliptic
estimate. 

In order to be able to apply again the subelliptic estimate to the
second summand above we need to use the formula
\begin{equation}
\label{eq:346}
\phi_{N}' D_{4}^{N} = \sum_{j=0}^{N-1} (-1)^{j} D_{4} \phi_{N}^{(j+1)}
D_{4}^{N-j-1} + (-1)^{N} \phi_{N}^{(N+1)}. 
\end{equation}
Thus, since $ \chi_{N}(N^{-1}D_{4}) $ commutes with $ D_{4}^{N} $, 
\begin{multline*}
\| x_{1}^{r-1} \phi_{N}'  \chi_{N}(N^{-1} D_{4}) D_{4}^{N} u \| 
\\
\leq
\sum_{j=0}^{N-1} \| X_{4} \phi_{N}^{(j+1)} \chi_{N}(N^{-1}D_{4})
D_{4}^{N - j -1} u \| + C \| \phi_{N}^{(N+1)} \chi_{N}(N^{-1} D_{4}) u
\| ,
\end{multline*}
where we used the fact that the field $ X_{4} $ could be reconstructed
by just ``pulling back'' one $ x_{4} $-derivative. A completely
analogous treatment leads to an analogous conclusion when $ j = 6 $:
\begin{multline*}
\| x_{2}^{q-1} \phi_{N}'  \chi_{N}(N^{-1} D_{4}) D_{4}^{N} u \| 
\\
\leq
\sum_{j=0}^{N-1} \| X_{6} \phi_{N}^{(j+1)} \chi_{N}(N^{-1}D_{4})
D_{4}^{N - j -1} u \| + C \| \phi_{N}^{(N+1)} \chi_{N}(N^{-1} D_{4}) u
\| .
\end{multline*}
Furthermore it is clear that the terms on the right of the above
inequalities yield a real analytic growth estimate, after using the
properties of $ \phi_{N} $.

\medskip
We are thus left with the term for $ j = 5 $:
\begin{multline*}
2 \left| \langle x_{2}^{p-1} \phi_{N}'  \chi_{N}(N^{-1} D_{4})
  D_{4}^{N} u , X_{5} \phi_{N}
\chi_{N}(N^{-1} D_{4}) D_{4}^{N} u \rangle \right|
\\
\leq
\delta \| X_{5} \phi_{N} \chi_{N}(N^{-1} D_{4}) D_{4}^{N} u \|^{2} +
\frac{1}{\delta} \| x_{2}^{p-1} \phi_{N}'  \chi_{N}(N^{-1} D_{4})
  D_{4}^{N} u  \|^{2}
\end{multline*}
Here, again, $ \delta $ is chosen so that the first term in the
r.h.s. above can be absorbed on the left of the subelliptic estimate,
as before. We just need to be concerned with the second term.
Contrary to what has been done before, pulling back one derivative is
of no help; we have to resort to the subelliptic part of the
subelliptic estimate, i.e. the $ 1/r $--Sobolev norm. To do this we
pull back $ D_{4}^{1/r} $. This is well defined since $ \xi_{4} > 1 $,
but is a pseudodifferential operator, and its commutator with $
\phi_{N} $ needs some care. Using Lemma \ref{lemma:commut} and Corollary
\ref{cor:2}, we have
\begin{multline*}
\| x_{2}^{p-1} \phi_{N}'  \chi_{N}(N^{-1} D_{4}) D_{4}^{N} u  \| 
\\
\leq
\| x_{2}^{p-1} \phi_{N}'  \chi_{N}(N^{-1} D_{4}) D_{4}^{N -
  \frac{1}{r}} u  \|_{\frac{1}{r}} 
\\
+
\sum_{k=1}^{N-1} c_{k} \| x_{2}^{p-1} \phi_{N}^{(k+1)}
\chi_{N}(N^{-1}D_{4}) D_{4}^{N-k} u \|
\\
+
\| x_{2}^{p-1} a_{N, N}(x, D) \chi_{N}(N^{-1}D_{4}) D_{4}^{N} u \|. 
\end{multline*}
The last term has analytic growth and we may forget about it. The
first on the r.h.s. above can be resubjected to the subelliptic
estimate and treated as we just did. Same argument for every summand
in the sum, except that here we have to pull back a $ D_{4}^{1/r} $
once more in order to use the subelliptic estimate. 

Hence applying the subelliptic estimate---Sobolev part---keeps
producing factors $ x_{2}^{p-1} $. This procedure terminates when the
exponent of $ x_{2} $ becomes greater or equal to $ q-1 $. At that
point we do not need the subelliptic part, but we are able to pull
back a whole derivative, thus decreasing the powers of $ x_{2} $.

Before writing the product of such an iteration, we discuss also the
terms from \eqref{eq:12comm} containing a double commutator. For $ j =
3, 4$
\begin{multline*}
\left| \langle  N^{-1} x_{1}^{r-1} \phi_{N}'' \chi_{N}(N^{-1}
D_{4}) D_{4}^{N} u , N x_{1}^{r-1} \phi_{N} \chi_{N}(N^{-1} D_{4}) D_{4}^{N} u
\rangle \right|
\\
\leq \frac{1}{2} \| N^{-1} x_{1}^{r-1} \phi_{N}'' \chi_{N}(N^{-1}
D_{4}) D_{4}^{N} u  \|^{2}  
\\
+
\frac{1}{2} \| N x_{1}^{r-1} \phi_{N} \chi_{N}(N^{-1} D_{4}) D_{4}^{N} u \|^{2}.
\end{multline*}
Each of the summands in the r.h.s. is then treates as before. Note
that $ N^{-1} \phi_{N}'' $ counts as a first derivative and so does $
N \phi_{N} $. Same argument for $ j = 5, 6 $.

\medskip
Iterating the above procedure we arrive at the inquality for $ j = 1,
\ldots, 6 $,
\begin{multline}
\label{eq:iteration}
\| X_{j} \phi_{N} \chi_{N}(N^{-1} D_{4}) D_{4}^{N} u  \|
\\[5pt]
\leq
\sum
\| x_{2}^{a_{5}(p-1) - b(q-1) - a_{2}} \phi_{N}^{(1+a_{3} + a_{4} +
  a_{5} + a_{6} + k_{1} + \cdots + k_{h})} 
\\
\cdot
\chi_{N}(N^{-1} D_{4})
D_{4}^{N - k_{1} - \cdots - k_{h} - \frac{a_{2} + a_{5}}{r} - a_{3} -
  a_{4} - a_{6} - b \frac{r-1}{r}} u \|
\\[5pt]
+ \sum
N^{-(a_{3} + a_{4} + a_{5} + a_{6})}
\| x_{2}^{a_{5}(p-1) - b(q-1) - a_{2}} \phi_{N}^{(1+ 2a_{3} + 2a_{4} +
  2a_{5} + 2a_{6} + k_{1} + \cdots + k_{h})} 
\\
\cdot
\chi_{N}(N^{-1} D_{4})
D_{4}^{N - k_{1} - \cdots - k_{h} - \frac{a_{2} + a_{5}}{r} - a_{3} -
  a_{4} - a_{6} - b \frac{r-1}{r}} u \|
\\[5pt]
+ \sum
N^{a_{3} + a_{4} + a_{5} + a_{6}}
\| x_{2}^{a_{5}(p-1) - b(q-1) - a_{2}} \phi_{N}^{(1 + k_{1} + \cdots + k_{h})} 
\\
\cdot
\chi_{N}(N^{-1} D_{4})
D_{4}^{N - k_{1} - \cdots - k_{h} - \frac{a_{2} + a_{5}}{r} - a_{3} -
  a_{4} - a_{6} - b \frac{r-1}{r}} u \| ,
\end{multline}
where each sum is taken on the indices $ a_{2}, \ldots, a_{6} $, $
k_{1}, \ldots, k_{h} $, such that 
\begin{equation}
\label{eq:cond1}
0 \leq N - k_{1} - \cdots - k_{h} - \frac{a_{2} + a_{5}}{r} - a_{3} -
  a_{4} - a_{6} - b \frac{r-1}{r} < 1 ,
\end{equation}
\begin{equation}
\label{eq:cond2}
0 \leq a_{5}(p-1) - b(q-1) - a_{2} < q-1 .
\end{equation}
We have to evaluate the supremum of the r.h.s. in the above
relation. From the second condition
$$ 
\frac{a_{5}(p-1) - a_{2}}{q-1} - 1 < b \leq \frac{a_{5}(p-1) - a_{2}}{q-1}
$$
so that from the first condition we deduce
$$ 
N - 1 <
\sum_{j=1}^{h} k_{j} +a_{2}\left( \frac{1}{r} - \frac{r-1}{r(q-1)}
\right) + a_{3} + a_{4} + a_{5} \frac{1}{s_{0}} + a_{6} < N +
\frac{r-1}{r} .
$$
Since $ r < q $ the coefficient of $ a_{2} $ is positive. In order to
estimate \eqref{eq:iteration} we need to compute
$$ 
\max \left\{ \sum_{j=1}^{h} k_{j} + a_{3} + a_{4} + a_{5} + a_{6}
\right\} ,
$$
where the maximum is on all indices verifying conditions
\eqref{eq:cond1}, \eqref{eq:cond2}. 
The three sums in \eqref{eq:iteration} give the same contribution,
because when an index is missing among the derivatives of $ \phi_{N} $
it is found at the exponent of the factor $ N $ and, viceversa when it
appears twice, it appears with a negative sign at the exponent of $ N $.

It is then clear that the maximum
is $ \sim N^{N s_{0}} $ so that we finally get
$$ 
\| X_{j} \phi_{N} \chi_{N}(N^{-1} D_{4}) D_{4}^{N} u  \| \leq C^{N+1} N!^{s_{0}},
$$
and this achieves the proof of the theorem.

\setcounter{section}{0}
\renewcommand\thesection{\Alph{section}}
\section{Appendix}

\setcounter{equation}{0}
\setcounter{theorem}{0}
\setcounter{proposition}{0}
\setcounter{lemma}{0}
\setcounter{corollary}{0}
\setcounter{definition}{0}
\renewcommand{\thetheorem}{\thesection.\arabic{theorem}}
\renewcommand{\theproposition}{\thesection.\arabic{proposition}}
\renewcommand{\thelemma}{\thesection.\arabic{lemma}}
\renewcommand{\thedefinition}{\thesection.\arabic{definition}}
\renewcommand{\thecorollary}{\thesection.\arabic{corollary}}
\renewcommand{\theequation}{\thesection.\arabic{equation}}
\renewcommand{\theremark}{\thesection.\arabic{remark}}

In this appendix we just include for the sake of completeness the
proofs of Lemma \ref{lemma:apriori} and Lemma \ref{lemma:vk}.

\begin{proof}[Proof of Lemma \ref{lemma:apriori}]
From the identity
$$ 
\| Q_{\mu} v \|^{2} = \mu^{4} \| v'' \|^{2} + \| V v \|^{2} - 2
\mu^{2} \re \langle v'' , V v \rangle,
$$
observing that
\begin{multline*}
2 \re \langle v'', Vv \rangle = \langle \left( V \partial^{2}
  + \partial^{2} V \right) v , v \rangle 
\\
= 2 \langle \partial
V \partial v ,v \rangle - \langle V' \partial v , v \rangle +
\langle \partial V' v , v \rangle
\\
= - 2 \langle V v' , v' \rangle + \langle \left[ \partial , V' \right] v
, v \rangle 
\\
=
-2 \langle V v' , v' \rangle + \langle  V'' v , v \rangle.
\end{multline*}
we deduce
$$ 
\| Q_{\mu} v \|^{2} = \mu^{2} \| v'' \|^{2}  + \| V v \|^{2} + 2
\mu^{2} \langle V v' , v' \rangle - \mu^{2} \langle V'' v , v \rangle.
$$
The next-to-last term above is non negative. Consider the last. Since
$ V $ is a polynomial potential we may write $ | V'' | \leq C_{V}
\left( V^{2} + 1 \right) $, for a suitable positive constant $ C_{V}
$. Then $ - \mu^{2} \langle V'' v , v \rangle \geq - C_{V} \mu^{2}
\left( \| V v \|^{2} + \| v \|^{2} \right) $ and \eqref{eq:apriori}
ensues. 
\end{proof}
\begin{proof}[Proof of Lemma \ref{lemma:vk}]
To start with $ v \in \mathscr{S}(\R_{x}) $ since the potential diverges at
infinity and $ v $ is an eigenfunction.

In order to prove \eqref{eq:vj} we are going to show that
\begin{equation}
\label{eq:vjL2}
\| v^{(j)} \| \leq C_{0, j} \mu^{-j/2}, 
\end{equation}
for $ j = 1, 2, 3 $.  The estimate \eqref{eq:vj} is then a
consequence of \eqref{eq:E}, \eqref{eq:v'}, \eqref{eq:vjL2} and the Sobolev immersion
theorem.

For the sake of simplicity let us write $ Q_{\mu}(x, \partial_{x}) $
as $ Q_{\mu}(x, \partial_{x}) = - \mu^{2} \partial^{2} + V(x) $, so
that $ v $ is such that $ Q_{\mu} v = E(\mu) v $. 
We may assume that $ v $ is normalized in $ L^{2}(\R) $.
Note that $ Q_{\mu}$ is self-adjoint. From the latter equation we
deduce that 
$$ 
E(\mu) = \langle Q_{\mu} v, v \rangle = \mu^{2} \| v'\|^{2} + \langle
V v, v \rangle \geq \mu^{2} \| v' \|^{2} ,
$$
which implies, by \eqref{eq:E},
$$ 
\| v' \| \leq C \frac{(E(\mu))^{1/2}}{\mu} \leq C_{1} \mu^{-1/2}.
$$
This proves \eqref{eq:vjL2} when $ j = 1 $.

From \eqref{eq:apriori}
$$ 
\mu^{4} \| v'' \|^{2} \leq C \left( \| Q_{\mu} v \|^{2} + \mu^{2} \| v
  \|^{2}\right) = C \left( E(\mu)^{2} + \mu^{2}\right),
$$
or
$$ 
\| v'' \| \leq C_{2} \mu^{-1},
$$
i.e. \eqref{eq:vjL2} with $ j = 2 $. 

Let us consider now the third derivative of $ v $. From $ Q_{\mu} v =
E(\mu) v $ we have $ Q_{\mu} v' = E(\mu) v' - V' v $. Hence
\begin{eqnarray*}
\mu^{2} \| v''' \| & \leq & C \left( \| Q_{\mu} v' \| + \mu \| v' \|
                            \right)
\\
& \leq & C \left( E(\mu) \| v' \| + \| V' v \| + \mu \| v' \| \right)
\\
& \leq & C' \mu^{1/2} + C \| V' v \|,
\end{eqnarray*}
where we used \eqref{eq:vjL2} for $ j = 1 $. 

Since, analogously to the above argument, 
$$
| V'| \leq C_{V}' \left( \mu^{- \frac{1}{2(2q-3)}} V +
  \mu^{\frac{1}{2}}\right), 
$$
we get that
$$ 
\|  V' v \| \leq C_{V}'' \left(  \mu^{- \frac{1}{2(2q-3)}} \| V v \| +
  \mu^{\frac{1}{2}} \right) 
$$
and hence
\begin{eqnarray*}
\mu^{2} \| v''' \| & \leq & C'' \mu^{\frac{1}{2}} +  C'''  \mu^{-
  \frac{1}{2(2q-3)}} \| V v \| 
\\
& \leq & C'' \mu^{\frac{1}{2}} +  C_{1} \mu^{- \frac{1}{2(2q-3)}}
         \left(  E(\mu) + \mu \right)
\\
&  \leq & C'' \mu^{\frac{1}{2}} + C_{2} \mu^{1 - \frac{1}{2(2q-3)}},
\end{eqnarray*}
which implies \eqref{eq:vjL2} when $ j = 3 $, since $ q > 3/2 $.
\end{proof}

\section{Appendix}

\setcounter{equation}{0}
\setcounter{theorem}{0}
\setcounter{proposition}{0}
\setcounter{lemma}{0}
\setcounter{corollary}{0}
\setcounter{definition}{0}
Let $ \phi_{N} $, $ N \in \N $, $ \omega_{N} $ and $ \chi_{N} $ be 
Ehrenpreis type cutoff functions, i.e. $ \phi_{N} \in
C_{0}^{\infty}(\R) $, $ \omega_{N}$, $ \chi_{N} \in C^{\infty}(\R) $, with $
\phi_{N} = 1 $ near the origin and $ \omega_{N} = 1 $ for $ x > 2 $, $
\omega_{N} = 0 $ for $ x < 1 $, $ \phi_{N} $, $ \omega_{N} $, $
\chi_{N} $ non negative. 
Furthermore we assume that $ \chi_{N} = 0 $ for $ x < 3 $ and $
\chi_{N} = 1 $ for $ x > 4 $, so that $ \omega_{N} \chi_{N} = \chi_{N}
$. 

Ehrenpreis type functions have the property that, for $ k \leq R N $,
$ R > s_{0} $,
$ | \partial^{k} \phi_{N}(x) | \leq C_{\phi}^{k+1} N^{k} $ and
analogous estimates for the derivatives of $ \omega_{N} $, $ \chi_{N} $.

\medskip

The purpose of this appendix is to give a proof of the %
\begin{lemma}
\label{lemma:commut}
Let $ 0 < \theta < 1 $. Then
\begin{multline}
\label{eq:commut}
[ \omega_{N}(N^{-1} D)  D^{\theta} , \phi_{N}(x)] \chi_{N}(N^{-1} D)
D^{N - \theta} 
\\
=
\sum_{k=1}^{N}  a_{N, k}(x, D) \chi_{N}(N^{-1} D) D^{N}, 
\end{multline}
where $ a_{N, k} $ is a pseudodifferential operator of order $ -k $
such that
\begin{equation}
\label{eq:aNk}
| \partial_{\xi}^{\alpha} a_{N, k}(x, \xi) | \leq C_{a}^{k+1} N^{k+\alpha}
\xi^{-k-\alpha}, \qquad 1 \leq k \leq N, \quad \alpha \leq N.
\end{equation}
\end{lemma}
\begin{proof}
From the composition formula of two pseudodifferential operators $ a
$, $ b $
$$ 
\sigma(a \circ b)(x, \xi) = \int e^{i z \zeta} a(x, \xi+\zeta)
b(x-z, \xi) dz \ \dbar \zeta ,
$$
we obtain
\begin{multline*}
\sigma\left( [ \omega_{N}(N^{-1} D)  D^{\theta} , \phi_{N}] \right) =
\sum_{k=1}^{N-1} \frac{1}{k!} \partial_{\xi}^{k}(\omega_{N}(N^{-1}
\xi) \xi^{\theta}) i^{-k} \phi_{N}^{(k)}(x) 
\\
+
\int \int_{0}^{1} e^{i z \zeta} (1-s)^{N-1} \frac{\zeta^{N}}{(N-1)!}
\left( \omega_{N}(N^{-1}\xi) \xi^{\theta}\right)^{(N)}(\xi + s \zeta)
\\
\cdot
\phi_{N}(x-z) ds dz \dbar \zeta.
\end{multline*}
Let us consider first the contribution to \eqref{eq:commut} given by
the terms in the sum in the first line of the above formula. 
We have
\begin{multline*}
\left |\sum_{k=1}^{N-1} \frac{1}{k!} \sum_{j=0}^{k} \binom{k}{j} 
   \omega_{N}^{(j)}(N^{-1}\xi) N^{-j} \theta (\theta - 1) 
\cdots (\theta - k + j + 1) \right.
\\
\left .
\vphantom{\sum_{k=1}^{N-1}}
\xi^{\theta - k + j}  \phi_{N}^{(k)}(x) \chi_{N}(N^{-1} \xi) \xi^{N-\theta} \right|
\\
=
\left |\sum_{k=1}^{N-1} \frac{1}{k!}
   \omega_{N}(N^{-1}\xi) \theta (\theta - 1) \cdots (\theta - k+1) 
\phi_{N}^{(k)}(x) \chi_{N}(N^{-1} \xi) \xi^{N-k} \right|
\\
=
\left |\sum_{k=1}^{N-1}
a_{N, k} (x, \xi)  \chi_{N}(N^{-1} \xi) \xi^{N} \right|
\\
\leq
C \sum_{k=1}^{N-1} \left(C_{\phi}^{k+1} \frac{1}{k!} N^{k} (k-1)! \right)
\chi_{N}(N^{-1} \xi) \xi^{N - k} ,
\end{multline*}
because $ \omega_{N}^{(j)} \chi_{N} = 0 $ if $ j \geq 1 $ and $
\omega_{N} \chi_{N} = \chi_{N} $.

Next we are going to examine the remainder
\begin{multline*}
\int \int_{0}^{1} e^{i z \zeta} (1-s)^{N-1} \frac{\zeta^{N}}{(N-1)!}
\left( \omega_{N}(N^{-1}\xi) \xi^{\theta}\right)^{(N)}(\xi + s \zeta)
\\
\cdot
\phi_{N}(x-z) ds dz \dbar \zeta .
\end{multline*}
In order to deal with the integral w.r.t. $ \zeta $ we partition the
whole $ \zeta $-space into the regions $ |\zeta| \leq \epsilon |\xi| $
and $ |\zeta| \geq \epsilon |\xi| $, where $ \epsilon $ denotes a
positive constant which will be chosen $ < 1 $. 
Let us consider the region $ |\zeta| \leq \epsilon |\xi| $ first. 
We point out that $ |\zeta| \leq \epsilon |\xi| $ implies $ (1 -
\epsilon) |\xi| \leq | \xi + s \zeta| \leq (1 + \epsilon) |\xi| $. 
Then 
\begin{multline*}
\left | \int \int_{|\zeta| \leq \epsilon |\xi|} \int_{0}^{1} e^{i z \zeta} (1-s)^{N-1} \frac{\zeta^{N}}{(N-1)!}
\left( \omega_{N}(N^{-1}\xi) \xi^{\theta}\right)^{(N)}(\xi + s \zeta) \right.
\\
\left . \vphantom{\int \int_{0}^{1}}
\cdot 
\phi_{N}(x-z) ds dz \dbar \zeta \right |
\\
=
\left | \int \int_{|\zeta| \leq \epsilon |\xi|} \int_{0}^{1} e^{i z \zeta}  \frac{(1-s)^{N-1}}{(N-1)!}
\left( \omega_{N}(N^{-1}\xi) \xi^{\theta}\right)^{(N)}(\xi + s \zeta) \right.
\\
\left . \vphantom{\int \int_{0}^{1}}
\cdot 
\phi_{N}^{(N)}(x-z) ds dz \dbar \zeta \right |
\\
\leq
\int \int_{|\zeta| \leq \epsilon |\xi|} \int_{0}^{1}  \frac{(1-s)^{N-1}}{(N-1)!}
\sum_{k=0}^{N} \binom{N}{k} N^{-k} | \omega_{N}^{(k)}(N^{-1}(\xi + s
\zeta)) | 
\\
\cdot
\theta | \theta -1| \cdots |\theta - N + k + 1| |\xi + s \zeta|^{\theta -
  N + k} |\phi_{N}^{(N)}(x-z)| ds dz \dbar \zeta
\\
\leq
C C_{\phi}^{N+1} \int_{0}^{1} \int \int_{|\zeta| \leq \epsilon |\xi|} (1-s)^{N-1} \frac{1}{(N-1)!}
\sum_{k=0}^{N} \binom{N}{k} C_{\omega}^{k+1}
\\
\cdot
(N - k - 1)! |\xi + s \zeta|^{k} N^{N}  dz \dbar \zeta ds |\xi|^{\theta -
  N}
\\
\leq
C_{1}^{N+1} N! |\xi|^{\theta -N} .
\end{multline*}
Here we used the fact that, on the support of a derivative of $
\omega_{N} $, $ \xi + s \zeta $ is equivalent to $ N $, so that for
every $ k = 0, 1, \ldots, N $, $ (\xi + s \zeta)^{k} \leq C^{k} N^{k}
$.

Let us consider the region $ |\zeta| \geq \epsilon |\xi| $. We have
\begin{multline*}
\int \int_{|\zeta| \geq \epsilon |\xi|} \int_{0}^{1} e^{i z \zeta} (1-s)^{N-1} \frac{\zeta^{N}}{(N-1)!}
\left( \omega_{N}(N^{-1}\xi) \xi^{\theta}\right)^{(N)}(\xi + s \zeta) 
\\
\cdot 
\phi_{N}(x-z) ds dz \dbar \zeta 
\\
=
 \int \int_{|\zeta| \geq \epsilon |\xi|} \int_{0}^{1} e^{i z \zeta}  \frac{(1-s)^{N-1}}{|\zeta|^{N+2}(N-1)!}
\left( \omega_{N}(N^{-1}\xi) \xi^{\theta}\right)^{(N)}(\xi + s \zeta) 
\\
\cdot 
\phi_{N}^{(2N+2)}(x-z) ds dz \dbar \zeta
\end{multline*}
As before the r.h.s. of the above relation can be estimated by

\begin{multline*}
\int \int_{|\zeta| \geq \epsilon |\xi|} \int_{0}^{1} (1-s)^{N-1}
\frac{1}{|\zeta|^{N+2} (N-1)!}
\sum_{k=0}^{N} \binom{N}{k} N^{-k} 
\\
\cdot
| \omega_{N}^{(k)}(N^{-1}(\xi + s \zeta)) |
\theta | \theta -1| \cdots |\theta - N + k + 1| |\xi + s \zeta|^{\theta -
  N + k} 
\\
\cdot
|\phi_{N}^{(2N+2)}(x-z)| ds dz \dbar \zeta
\end{multline*}
We consider the above quantity in the cases $ 0 \leq k \leq
N-1$ and $ k = N $. In the latter case we observe that
\begin{multline*}
\int_{x - \supp \phi_{N}} \int_{\substack{|\zeta| \geq \epsilon |\xi|
    \\ N \leq |\xi + s \zeta| \leq 2 N }} \int_{0}^{1} (1-s)^{N-1}
\frac{1}{|\zeta|^{2} (N-1)!}
\frac{C_{\omega}^{N+1}}{\epsilon^{N}} 
\\
|\xi + s \zeta|^{\theta} \frac{1}{|\xi|^{N}}
C_{1,\phi}^{2N + 3} N^{2N} ds dz \dbar \zeta
\\
\leq
C_{1}^{N+1}  N^{N} ( N^{\theta} |\xi|^{-N} ) ,
\end{multline*}
where the factor $ |\zeta|^{-2} $ makes the $ \zeta $-integral
convergent and $ C_{1} $ is a suitable positive constant.

Let us consider the remaining cases: $ 0 \leq k \leq N-1 $:
\begin{multline*}
\int \int_{|\zeta| \geq \epsilon |\xi|} \int_{0}^{1} (1-s)^{N-1}
\frac{1}{|\zeta|^{N+2} (N-1)!}
\sum_{k=0}^{N-1} \binom{N}{k} N^{-k} 
\\
\cdot
| \omega_{N}^{(k)}(N^{-1}(\xi + s \zeta)) |
\theta | \theta -1| \cdots |\theta - N + k + 1| |\xi + s \zeta|^{\theta -
  N + k} 
\\
\cdot
|\phi_{N}^{(2N+2)}(x-z)| ds dz \dbar \zeta
\\
\leq
\int_{x - \supp \phi_{N}} \int_{\substack{|\zeta| \geq \epsilon |\xi|\\ \xi + s \zeta \geq
    N}} \int_{0}^{1}  
\frac{1}{|\zeta|^{2} (N-1)!}
\sum_{k=0}^{N-1} \binom{N}{k} C_{\omega}^{k+1}
\\
\cdot
(N - k - 1)! N^{\theta -
  N + k} 
C_{\phi}^{2N+3} N^{2N}  ds dz \dbar \zeta  (\epsilon^{-N} |\xi|^{-N})
\\
\leq
C_{2}^{N+1}  N^{N} ( N^{\theta} |\xi|^{-N} ) .
\end{multline*}
The two estimates in the regions $ |\zeta| \lesseqgtr \epsilon |\xi| $
imply that
$$ 
\sigma\left( [ \omega_{N}(N^{-1} D)  D^{\theta} , \phi_{N}] \right)
\chi_{N}(\xi) \xi^{N-\theta} = \sum_{k=1}^{N} a_{N, k}(x, \xi)
\chi_{N}(\xi) \xi^{N} ,
$$
where $ a_{N, k} $ satisfies the estimate \eqref{eq:aNk} when $ \alpha
= 0$. The estimate of the derivatives w.r.t. $ \xi $ in \eqref{eq:aNk}
are proved analogously. 
\end{proof}
As a byproduct of the proof of the above lemma, keeping in mind the
definition of the cutoff $ \omega_{N} $, $ \chi_{N} $, we have the 
\begin{corollary}
\label{cor:2}
For $ 1 \leq k \leq N-1 $ in \eqref{eq:aNk} we have that
\begin{multline}
\label{eq:aN-1k}
a_{N,k}(x, D) \chi_{N}(N^{-1}D) D^{N} 
\\
= \frac{\theta (\theta - 1) \cdots (\theta - k +1)}{k!} D^{k}_{x}\phi_{N}(x)
\chi_{N}(N^{-1}D) D^{N-k} .
\end{multline}
\end{corollary}

\end{document}